\documentclass[12pt,a4paper]{amsart}
\usepackage[totalwidth=15.75cm,totalheight=22.275cm]{geometry}
\usepackage[utf8]{inputenc}
\usepackage[font=small]{caption}
\usepackage{mathscinet}
\usepackage{amsmath}     
\usepackage{amssymb}     
\usepackage{amsfonts}     
\usepackage{amsthm}     
\usepackage[showonlyrefs]{mathtools}
\mathtoolsset{showonlyrefs}
\usepackage{hyperref}
\numberwithin{equation}{section}
\newtheorem{theorem}{Theorem}[section]
\newtheorem{lemma}[theorem]{Lemma}
\theoremstyle{definition}

\theoremstyle{remark}

\newtheorem*{ack}{Acknowledgment}
\newtheorem*{rem}{Remark}

\def\R{\mathbb{R}}
\def\C{\mathbb{C}}
\def\N{\mathbb{N}}

\def\LO{O}

\title[Differential equations and completely regular growth]{A question of Gol\cprime dberg and Ostrovskii
concerning linear differential equations with coefficients of completely regular growth}
\author{Walter Bergweiler}
\address{Mathematisches Seminar, Christian--Albrechts--Uni\-versi\-t\"at zu Kiel, 
24098 Kiel, Germany}
\email{bergweiler@math.uni-kiel.de}
\dedicatory{Dedicated to the memory of Anatoly A. Gol'dberg and Iossif V. Ostrovskii}
\makeatletter
\@namedef{subjclassname@2020}{%
  \textup{2020} Mathematics Subject Classification}
\makeatother

\subjclass[2020]{34M10, 34M05, 34M03, 30D15}
\keywords{Entire function, linear differential equation, 
completely regular growth}

\begin{document}

\begin{abstract}
We show that a linear differential equation whose coefficients are entire functions of
completely regular growth may have an entire solution of finite order which is not of
completely regular growth. This answers a question of Gol\cprime dberg and Ostrovskii.
\end{abstract}

\maketitle

\section{Introduction and results} \label{intro}
An entire function $f$ of order $\rho$ is said to be of 
\emph{completely regular growth} in the sense of Levin and Pfluger if there exists
a $2\pi$-periodic function $h\colon\R\to\R$ which does not vanish identically
 such  that
\begin{equation}\label{0a}
\log|f(re^{i\theta})|=h(\theta)r^\rho +o(r^\rho)
\end{equation}
as $r\to\infty$, for $re^{i\theta}$ outside a union of disks $\{z\colon |z-z_j|<r_j\}$ satisfying
\begin{equation}\label{0b}
\sum_{|z_j|\leq r} r_j=o(r)
\end{equation}
as $r\to\infty$. 
For a thorough treatment of entire functions of completely regular growth we refer to~\cite{Levin1980}.

The function $h$ is called the \emph{indicator} of~$f$.
The indicator is trigonometrically convex~\cite[Chapter~I, Lemma~6]{Levin1980}.
Conversely, for every trigonometrically convex function there exists an entire function 
of completely regular growth which has this function as its indicator \cite[Chapter~II, Theorem~3]{Levin1980}.

One may replace the term $r^\rho$ in~\eqref{0a} by $r^{\rho(r)}$ with a proximate order $\rho(r)$.
We refer to~\cite[Chapter~I, \S~12]{Levin1980}
for the definition of a proximate order.

If $n\in\N$ and  $a_0,\dots,a_{n-1}$ are polynomials, then every solution
$f$ of the linear differential equation 
\begin{equation}\label{0c}
f^{(n)}+a_{n-1}f^{(n-1)}+\ldots+a_1f'+a_0f=0
\end{equation}
is an entire function of completely regular growth. According to~\cite[Problem 16.2]{Havin1994},
this is due to Petrenko~\cite[Chapter IV]{Petrenko1984}.
It also follows from the work of Steinmetz
(\cite{Steinmetz1989}; see also~\cite[Section~2]{Steinmetz1993}).

If the coefficients $a_j$ are transcendental, then the general solution of~\eqref{0c}
is of infinite order~\cite[Satz~1]{Wittich1966}.
However, even if~\eqref{0c} has a solution of finite order,
this solution need not be of completely regular growth. In fact, 
Gol\cprime dberg~\cite[Problem 16.2]{Havin1994} showed that every entire function is 
a solution of some differential equation of the form~\eqref{0c} with entire
coefficients $a_0,\dots,a_{n-1}$.

Gol\cprime dberg and Ostrovskii (\cite[Problem 16.2]{Havin1994}, \cite[Question 5.5]{Gundersen2017})
asked whether a transcendental entire solution $f$ of~\eqref{0c}
of finite order must be of completely regular growth if the coefficients 
$a_0,\dots,a_{n-1}$ are of completely regular growth.
For some further discussion and results concerning this problem we refer
to~\cite{Bandura2020,Heittokangas2015,Wen2018}.

We will show that the answer to the question of Gol\cprime dberg and Ostrovskii is negative.
\begin{theorem}\label{theorem1}
Let $(z_k)$ be a sequence of distinct points in $\C\setminus\{0\}$ which tends to~$\infty$.
Suppose that the exponent of convergence $\sigma$ of $(z_k)$ is less than $1$
so that 
\begin{equation}\label{1a}
f(z)=\prod_{k=1}^\infty \left(1-\frac{z}{z_k}\right)
\end{equation}
defines an entire function $f$ of order $\sigma$.

Suppose also that there exists $C>0$ such that 
\begin{equation}\label{1b}
\left|\frac{f''(z_k)}{f'(z_k)^2}\right|\leq C
\end{equation}
for all $k\in\N$.

Then there exist entire functions $A_0$ and $B_0$
of order at most $\sigma$ such that 
\begin{equation}\label{1c}
f''+A_0f'+B_0 f=0.
\end{equation}
Moreover, given $\rho$ greater than $\sigma$ there exist entire functions $A$ and $B$ 
of order $\rho$ and of completely regular growth such that 
\begin{equation}\label{1d}
f''+Af'+B f=0.
\end{equation}
\end{theorem}
To answer the question by 
Gol\cprime dberg and Ostrovskii it thus suffices to show that there exists an
entire function $f$ which satisfies the hypotheses of Theorem~\ref{theorem1}
but which is not of completely regular growth.

One way to check that condition~\eqref{1b} is satisfied for a given function $f$ is to establish 
a lower bound for $|f'|$ in some disk around $z_k$. Since $f''/(f')^2=-(1/f')'$ 
one can then obtain an upper bound for $|f''(z_k)/f'(z_k)^2|$ from Cauchy's
integral formula.

\begin{theorem}\label{theorem2}
Let $0<\rho<1$, let $(r_k)$ be an increasing sequence of positive numbers 
tending to infinity and let $(n_k)$ be a sequence of natural numbers 
satisfying $n_k\sim r_k^\rho$ as $k\to\infty$.

If $(r_k)$ tends to $\infty$ sufficiently fast, then 
\begin{equation}\label{1e}
f(z)=\prod_{k=1}^\infty \left(1-\left(\frac{z}{r_k}\right)^{n_k}\right)
\end{equation}
defines an entire function $f$ of order $\rho$ and lower order $0$
which satisfies the hypotheses of Theorem~\ref{theorem1}.
\end{theorem}
Since the order and lower order of the function $f$ given in Theorem~\ref{theorem2}
are different, this function cannot be of completely regular growth.
As noted above, combined with Theorem~\ref{theorem1}
this solves the problem posed by Gol\cprime dberg and Ostrovskii.

\begin{ack}
I thank Alexandre Eremenko and the referee for helpful comments.
\end{ack}
\section{Proofs of Theorems~\ref{theorem1} and~\ref{theorem2}} \label{proof1}
A standard interpolation result says that given sequences $(z_k)$ and
$(w_k)$ in $\C$, with $z_k\to\infty$, there
exists an entire function $h$ satisfying $h(z_k)=w_k$.
This result was also used by Gol\cprime dberg~\cite{Havin1994} in his proof that every
entire function $f$ satisfies an equation of the form~\eqref{0c} with
entire coefficients $a_0,\dots,a_{n-1}$.

The standard proof of  this interpolation result is as follows. 
Let $f$ be an entire function with simple zeros at the points $z_k$ and
let $g$ be a meromorphic function having simple poles at the $z_k$ with principal
part $u_k/(z-z_k)$, where $u_k=w_k/f'(z_k)$. Then $h:=fg$ has the required 
property. 

We will control the growth of $g$ and hence $h$
using the following lemma~\cite[Chapter~5, Theorem 6.1]{Goldberg2008}.
Here and in the following we use the standard terminology of Nevanlinna theory
as given in~\cite{Goldberg2008}.

\begin{lemma}\label{lemma1}
Let $(z_k)$ and $(u_k)$ be sequences in $\C\setminus\{0\}$, with $z_k\to\infty$ and
\begin{equation}\label{3x}
\sum_{k=1}^\infty \left|\frac{u_k}{z_k}\right|<\infty.
\end{equation}
Then 
\begin{equation}\label{3a}
g(z)=\sum_{k=1}^\infty \frac{u_k}{z-z_k}
\end{equation}
defines a meromorphic function~$g$
which satisfies $m(r,g)=o(1)$ as $r\to\infty$.
\end{lemma}

\begin{proof}[Proof of Theorem~\ref{theorem1}]
Put $u_k=-f''(z_k)/f'(z_k)^2$.
Since the exponent of convergence of $(z_k)$ is less than $1$
we deduce from~\eqref{1b} that~\eqref{3x} holds.
Let $g$ be defined by~\eqref{3a} and put $A_0=fg$. Then
\begin{equation}\label{3f}
A_0(z_k)=u_k f'(z_k)=-\frac{f''(z_k)}{f'(z_k)}
\end{equation}
for all $k\in\N$.
It follows that $f''(z_k)+A_0(z_k)f'(z_k)=0$ for all $k\in\N$ and hence that 
the function $B_0$ defined by 
\begin{equation}\label{3g}
B_0=-\frac{f''+A_0f'}{f}
\end{equation}
is entire. Thus~\eqref{1c} is satisfied.

Moreover, it follows from Lemma~\ref{lemma1} that 
\begin{equation}\label{3h}
T(r,g)=N(r,g)+o(1)=N\!\left(r,\frac{1}{f}\right)+o(1)\leq T(r,f)+\LO(1) .
\end{equation}
Thus the order of $g$ does not exceed that of~$f$, which is equal to~$\sigma$.
It now follows from the definition of $A_0$ and $B_0$ that their
orders are also not larger than~$\sigma$.

To prove the existence of entire functions $A$ and $B$ with the properties
stated, let $H$ be any entire function which has completely regular 
growth of order $\rho$. Since $\rho>\sigma$,
it follows from~\cite[Chapter III, \S~4]{Levin1980} that $Hf$ and $Hf'$ are 
of completely regular growth, with the same indicator as~$H$.
We may choose $H$ such that the indicator of $H$ is positive.
(For example, we may choose $H$ such that the indicator $h$ of $H$ satisfies
$h(\theta)\equiv c$ for some $c>0$, noting that the function $h$ defined this
way is trigonometrically convex.)
It then follows that $A:=A_0+Hf$ and $B:=B_0-Hf'$ are also of completely regular 
growth, again with the same indicator as~$H$.
Finally, \eqref{1c} yields that $f$, $A$ and $B$ satisfy~\eqref{1d}.
\end{proof}
\begin{rem}
If the indicator of $H$ is not positive, then the functions $A$ and $B$ defined in the
above proof need not be of completely regular growth.
However, 
the results in~\cite{Goldberg1981,Goldberg1982,Goldberg1983} imply in particular that there 
exists $c\in\C\setminus\{0\}$ such that $A:=A_0+cHf$ and $B:=B_0-cHf'$ have completely regular growth.
And~\eqref{1d} is also satisfied for these functions.
\end{rem}

\begin{proof}[Proof of Theorem~\ref{theorem2}]
It is not difficult to see that if $(r_k)$ tends to $\infty$ sufficiently
fast, then~\eqref{1e} defines an entire function $f$ of order $\rho$ and lower
order~$0$. Moreover,  letting $(r_k)$ tend to $\infty$ sufficiently
fast we can achieve that if $z\in\C$ with $|z|\sim r_k$, then
\begin{equation}\label{2a}
f(z)\sim  \prod_{j=1}^{k}  \left(1-\left(\frac{z}{r_j}\right)^{n_j}\right)
\sim \left(1-\left(\frac{z}{r_k}\right)^{n_k}\right)
(-1)^{k-1} \prod_{j=1}^{k-1}  \left(\frac{z}{r_j}\right)^{n_j}
\end{equation}
and
\begin{equation}\label{2c}
\frac{zf'(z)}{f(z)}
=\sum_{j=1}^\infty n_j \frac{\left(\frac{z}{r_j}\right)^{n_j}}{\left(\frac{z}{r_j}\right)^{n_j}-1}
=\sum_{j=1}^{k-1} n_j  + n_k\frac{\left(\frac{z}{r_k}\right)^{n_k}}{\left(\frac{z}{r_k}\right)^{n_k}-1}+ o(1)
\end{equation}
as $k\to\infty$.

A zero $\xi$ of $f$ has the form $\xi=\omega_k r_k$ where $\omega_k^{n_k}=1$. For $z=\omega_k r_k(1+\zeta/n_k)$ with
$|\zeta|\leq 1$ we have $(z/r_k)^{n_k}\to e^\zeta$ as $k\to\infty$.
The disk $\{\zeta\colon |\zeta|\leq 1\}$ corresponds to 
the disk $D_\xi:=\{z\colon |z-\xi|\leq r_k/n_k\}$.
Since we may assume that $\sum_{j=1}^{k-1}n_j=o(n_k)$
it follows from~\eqref{2c} that 
for large $k$ the disk $D_\xi$ contains no zero of~$f'$.

We deduce from~\eqref{2a} and~\eqref{2c} that if $z\in\partial D_\xi$, then
\begin{equation}\label{2e}
|f'(z)|= \frac{1}{|z|}\cdot |f(z)| \cdot \left|\frac{zf'(z)}{f(z)}\right|
\sim 
\frac{1}{r_k}\cdot \prod_{j=1}^{k-1} \left(\frac{r_k}{r_j}\right)^{n_j}  \cdot n_k\, |e^\zeta|.
\end{equation}
Hence 
\begin{equation}\label{2f}
\left|\frac{f''(\xi)}{f'(\xi)^2}\right|
=\left| \frac{1}{2\pi i}\int_{\partial D_\xi} \frac{dz}{f'(z)(z-\xi)^2}\right|
\leq 
\frac{n_k}{r_k} \max_{z\in \partial D_\xi}\frac{1}{|f'(z)|}
\leq 
(1+o(1)) e \prod_{j=1}^{k-1} \left(\frac{r_j}{r_k}\right)^{n_j}=o(1)
\end{equation}
as $k\to\infty$.
Thus $f$ satisfies the hypothesis of Theorem~\ref{theorem1}.
\end{proof}

\end{document}